\documentclass[reqno,12pt]{amsart}
\usepackage{amscd,amsmath,amssymb}
\theoremstyle{plain}
\newtheorem{thm}{Theorem}[section]
\theoremstyle{plain}
\newtheorem{lemma}[thm]{Lemma} 
\newtheorem{proposition}[thm]{Proposition}

\begin{document}
\title{ Some Remarks On Noncommutative Instantons } \par
\author{ Nikolay A. Ivanov }

\address{\hskip-\parindent
Nikolay A. Ivanov  \\
Department of Mathematics and Informatics\\
Veliko Tarnovo University \\}
\email{Nikolay.Antonov.Ivanov@gmail.com}

\begin{abstract}
We make some comments on noncommutative $U(N)$-instantons on $\mathbb{R}^4_{\theta}$. We elaborate on the equations 
for the ASD-connection for free modules. Further we make some remarks on the computation of the topological index of 
ADHM instantons. 
\end{abstract}

\keywords{Noncommutative $\mathbb{R}^4$; ADHM-construction; $U(N)$-connection, topological index.}

\maketitle

\section{Introduction} \label{sec1}

The idea of noncommutative field theories is old (see Snyder \cite{Snyder47}). Much later there had been 
significant interest motivated by Operator algebra theory (from Mathematics side) and by String theory 
(from Physics side). Connes and Rieffel \cite{CR87, C94} studied the space of minima of a version of 
Yang-Mills functional 
on noncommutative two-tori. This YM-functional is defined by considering a connection on a projective module, 
taking the corresponding curvature two-form, collapsing it into itself and taking trace. 
The usual Yang-Mills functional is constructed by integrating the wedge product of a curvature forms with itself over 
the underlying four dimensional space. Connes, Douglas and Schwarz \cite{CDS98} considered Matrix theory, 
compactified on tori; Seiberg and Witten \cite{SW99} considered a B-field in the presence of a $Dp$-brane 
(which prevents the B-field to be gauged out). Noncommutative field theory has become a separate theory, 
having its mathematical \cite{C80, C94} and physical \cite{DN01, S03} inclinations. 
\par
Conventional instantons are topological field configurations of gauge fields. They have been introduced by 
Belavin, Polyakov, Schwarz and Tyupkin \cite{BPST75} and studied extensively further. 
Noncommutative instantons are an analog pertaining to the Noncommutative field theory. 
They have been introduced by Nekrasov and Schwatz \cite{NS98} and developed further \cite{N00, S01, KS01, S02, TZS02}. 
There are more recent discussions on the topic. \cite{S10, HN13}

\section{Noncommutative Spaces, Algebras of Operators and the Fock Hilbert-Space} \label{sec2}

Consider the space $\mathbb{R}_{\theta}^4$ with four coordinate operators 
$\hat{x}_1, \hat{x}_2, \hat{x}_3, \hat{x}_4$ satisfying the relations $[\hat{x}_i,\hat{x}_j] = i\theta_{ij},$ 
where $\theta_{ij}$ is a real nondegenerate antisymmetric matrix. After $\rm{SO(4)}$ rotation of coordinates, 
$\theta$ can be brought to the form:
$$\theta_{ij} =
\left(\!\!\begin{array}{cccc}0 & \theta_{12} & 0 & 0 \\ -\theta_{12} & 0 & 0 & 0 \\
0 & 0 & 0 & \theta_{34} \\ 0 & 0 & -\theta_{34} & 0 \end{array}\!\!\right),$$
where $\theta_{12} > 0,\ \theta_{12} + \theta_{34} \geq 0.$ 
\par
A theory built up by such operators is nonlocal: $\sigma_{\hat{x}_1} \cdot \sigma_{\hat{x}_2} \geq 
\frac{1}{2} | \theta_{12} |$; $\sigma_{\hat{x}_3} \cdot \sigma_{\hat{x}_4} \geq \frac{1}{2} |\theta_{34} |.$ Here 
$\sigma_{\hat{x}_i} = \sqrt{\langle \hat{x}_i^2 \rangle - \langle \hat{x}_i \rangle^2}.$
\par
Since the pairs of operators $(\hat{x}_1, \hat{x}_2)$ and $(\hat{x}_3, \hat{x}_4)$ satisfy the canonical commutation 
relations, by the Stone-von Neumann theorem they can be analytically represented as the position and momentum 
operators on ${\rm L}^2(\mathbb{R}^2):$ 
\begin{equation*} 
\hat{x}_1 \psi(x_1,x_3) = x_1 \psi(x_1,x_3),\ \  \hat{x}_2 \psi(x_1,x_3) = -i \theta_{12} \frac{d \psi(x_1,x_3)}{dx_1} 
\end{equation*}
\begin{equation} \label{canonical}
\hat{x}_3 \psi(x_1,x_3) = x_3 \psi(x_1,x_3),\ \  \hat{x}_4 \psi(x_1,x_3) = -i \theta_{34} \frac{d \psi(x_1,x_3)}{dx_3}.
\end{equation}

One can consider appropriate algebras $\mathcal{A}$ on $\mathbb{R}_{\theta}^4$. To start, consider these examples
\cite{S01, KS01}: \\
$\bullet$ ${\displaystyle \mathcal{S}(\mathbb{R}^4) = \{ f \in C^{\infty}(\mathbb{R}^4) \ | \ 
\underset{x \in \mathbb{R}^4}{\rm{sup}} | x^{\alpha} D^{\beta} f(x) | < \infty,\ 
\forall \alpha, \beta \in \mathbb{N}_0^4 \}}$; \\
$\bullet$ ${\displaystyle \Gamma^m_{\rho}(\mathbb{R}^4) = \{ f \in C^{\infty}(\mathbb{R}^4) \ | \ 
| \partial_{\alpha} f(x) | \leq C_{\alpha} (1+ |x|^2)^{\frac{1}{2}(m -\rho | \alpha |)},\ 
\forall \alpha \in \mathbb{N}_0^4 \} }$, \\
where $m < 0,\ 0< \rho \leq 1,\ {\rm and}\ | \alpha | = \alpha_1 + \alpha_2 + \alpha_3 + \alpha_4;$ \\
$\bullet$ ${\displaystyle \mathcal{K}^{\infty}(\mathbb{R}^4) = \{ f \in C^{\infty}(\mathbb{R}^4) \ | \ 
\underset{|x| \to \infty}{\rm lim}| \partial_{\alpha} f(x) | = 0,\ \forall \alpha \in \mathbb{N}_0^4 \} }$. \\
Next we use the Weyl formula which assigns to a function on $\mathbb{R}^4$ an operator \cite{S03} on 
$\mathbb{R}^4_{\theta}$:  \\
${\displaystyle f(x) \mapsto \hat{\mathcal{W}}[f(x)] = \int \ d^4 x \ f(x) \hat{\Delta}(x),\ \hat{\Delta}(x) = 
\frac{1}{(2\pi)^4} \int \ d^4 k \ {\rm e}^{ik^j \hat{x}_j} {\rm e}^{-ik^j x_j}}$. 
We have $\hat{\mathcal{W}}[{\rm e}^{ik^j x_j}] = {\rm e}^{ik^j \hat{x}_j}$ and the reverse formula is: 
\par
${\displaystyle {\rm Tr}(\hat{\mathcal{W}}[f] \hat{\Delta}(x)) = f(x)}$. \\
Denote: ${\displaystyle \mathcal{S} = \hat{\mathcal{W}}[\mathcal{S}(\mathbb{R}^4)],\ 
\Gamma^m_{\rho} = \hat{\mathcal{W}}[\Gamma^m_{\rho}(\mathbb{R}^4)],\ 
\mathcal{K}^{\infty} = \hat{\mathcal{W}}[\mathcal{K}^{\infty}(\mathbb{R}^4)].}$
\par
We note that:
\begin{lemma}
$\mathcal{S}$, $\Gamma^m_{\rho}$ ($m<0,\ 0 < \rho \leq 1$) and $\mathcal{K}^{\infty}$ are norm-dense subalgebras of 
the algebra of compact operators $\mathcal{K}$ on $L^2(\mathbb{R}^2)$.
\end{lemma}

\begin{proof}
It is well known that $\mathcal{S}$ is a norm dense subalgebra of $\mathcal{K}$ and it is known that all these sets 
are algebras, containing $\mathcal{S}$. \cite{S01, KS01}. It is enough to show that they are all subalgebras of 
$\mathcal{K}$.
\par
According to the Baker-Campbell-Hausdorff Formula we have \cite{H13} :
\begin{equation*}
[{\rm e}^{ik^j \hat{x}_j} \psi](x_1,x_3) = 
{\rm e}^{i(k^1 k^2 \theta_{12} + k^3 k^4 \theta_{34})/2} 
{\rm e}^{i(k^1 x_1 + k^3 x_3)} \psi(x_1 + \theta_{12} k^2, x_3 + \theta_{34} k^4).
\end{equation*}
Suppose that $f(x_1,x_3) \in C_0(\mathbb{R}^2)$ is a function from one of the above algebras (in particular it 
goes to zero at infinity). Let $\{ \psi_n(x_1,x_3) \}_{n=1}^{\infty}$ be a sequence from the unit ball of 
$L^2(\mathbb{R}^2)$. We have:
\begin{equation*}
\hat{\mathcal{W}}[f] \psi_n(x_1,x_3) = \int \ d^4 y \ f(y) \frac{1}{(2\pi)^4} \int \ d^4 k \ {\rm e}^{ik^j \hat{x}_j} 
{\rm e}^{-ik^j y_j} \psi_n(x_1,x_3) = 
\end{equation*}
\begin{equation*}
\frac{1}{(2\pi)^4} \int \int d^4 y d^4 k f(y) {\rm e}^{-ik^j y_j} {\rm e}^{i(k^1 k^2 \theta_{12} + 
k^3 k^4 \theta_{34})/2} {\rm e}^{i(k^1 x_1 + k^3 x_3)} .
\end{equation*}
\begin{equation*}
. \psi_n(x_1 + \theta_{12} k^2, x_3 + \theta_{34} k^4) = 
\end{equation*}
\begin{equation*}
\frac{1}{(2\pi)^4} \int \int d^4 y dk^2 dk^4 f(y) \delta (k^2 \theta_{12} + x_1 - y_1) 
\delta (k^4 \theta_{34} + x_3 - y_3) . 
\end{equation*}
\begin{equation*}
. {\rm e}^{-i(k^2 y_2 + k^4 y_4)} \psi_n(x_1 + \theta_{12} k^2, x_3 + \theta_{34} k^4) = 
\end{equation*}
\begin{equation*}
\frac{1}{(2\pi)^4} \int \int dy_2 dy_4 dk^2 dk^4 f(x_1 + \theta_{12} k^2, y_2, x_3 + \theta_{34} k^4, y_4) .
\end{equation*}
\begin{equation*}
. {\rm e}^{-i(k^2 y_2 + k^4 y_4)} \psi_n(x_1 + \theta_{12} k^2, x_3 + \theta_{34} k^4) = 
\end{equation*}
\begin{equation*}
\frac{1}{(2\pi)^4 \theta_{12} \theta_{34}} \int \int d^4 y 
{\rm e}^{-i[y_2(y_1 -x_1)/\theta_{12}  + y_4(y_3 -x_3)/\theta_{34}]} 
f(y_1, y_2, y_3, y_4) \psi_n(y_1, y_3).
\end{equation*}
For $y_1 - x_1$ and $y_3 - x_3$ both away from zero there is dampening coming from $f$ and the rapidly oscillating 
phase. Thus the volume is concentrated around the origin. Compactness argument shows that we can find a fundamental 
subsequence, hence the operator is compact.
\end{proof}

To an algebra $\mathcal{A}$ of the above form we can adjoint a unit. The algebra we obtain in this way we will denote 
by $\tilde{\mathcal{A}}$. It follows from the above lemma that the $K$-theories of all these algebras coinside: If 
$\mathcal{A} = \mathcal{K}^{\infty}, \ \Gamma^m_{\rho}, \ {\rm or} \ \mathcal{S}$, then
\begin{equation*}
K_0(\mathcal{A}) = K_0(\mathcal{K}) = \mathbb{Z}{\rm ,} \ \ K_1(\mathcal{A}) = K_1(\mathcal{K}) = \{ 0 \}{\rm ,} 
\end{equation*}
\begin{equation} \label{K}
K_0(\tilde{\mathcal{A}}) = K_0(\tilde{\mathcal{K}}) = \mathbb{Z}^2 \ {\rm and} \  
K_1(\tilde{\mathcal{A}}) = K_1(\tilde{\mathcal{K}}) = \{ 0 \}.
\end{equation}
$\mathcal{A}$ is closed with respect to taking derivatives, which are defined in the following way:
\begin{equation*}
\partial_1(\hat{f}) \equiv \frac{i}{\theta_{12}} [\hat{x}_2,\hat{f}], \  \  \ 
\partial_2(\hat{f}) \equiv \frac{-i}{\theta_{12}} \ [\hat{x}_1,\hat{f}],
\end{equation*}
\begin{equation*} 
\partial_3(\hat{f}) \equiv \frac{i}{\theta_{34}} [\hat{x}_4,\hat{f}], \  \  \ 
\partial_4(\hat{f}) \equiv \frac{-i}{\theta_{34}} [\hat{x}_3,\hat{f}].
\end{equation*}
Note that these are all inner derivations and we have $\partial_1(\hat{x}_1^n) = n \hat{x}_1^{n-1}$, etc.
\par
The Lie algebra $\mathbb{R}^4$ acts on $\mathcal{A}$ via the derivatives. 
The derivatives commute, due to the Bianchi identity. 
\par
We have two commuting sets of operators, satisfying the canonical commutation relations:
\begin{equation*}
c_1 \equiv \frac{1}{\sqrt{2\theta_{12}}} \hat{z_1} = \frac{1}{\sqrt{2\theta_{12}}} (\hat{x_1} + i\hat{x_2}),\ \ 
c_1^* \equiv \frac{1}{\sqrt{2\theta_{12}}} \hat{z_1}^* = \frac{1}{\sqrt{2\theta_{12}}} (\hat{x_1} - i\hat{x_2}), 
\end{equation*}
\begin{equation*}
c_2 \equiv \frac{1}{\sqrt{2\theta_{34}}} \hat{z_2} = \frac{1}{\sqrt{2\theta_{34}}} (\hat{x_3} + i\hat{x_4}),\ \ 
c_2^* \equiv \frac{1}{\sqrt{2\theta_{34}}} \hat{z_2}^* = \frac{1}{\sqrt{2\theta_{34}}} (\hat{x_3} - i\hat{x_4}).
\end{equation*}
The commutation relations are:
\begin{equation*} 
[c_1,c_2] = [c_1,c_2^*] = [c_1^*,c_2] = [c_1^*,c_2^*] = 0,\ [c_1,c_1^*] = [c_2,c_2^*] = I.
\end{equation*}
The Hilbert space on which these operators are represented is $\mathcal{F} = \mathcal{F}_1 \otimes \mathcal{F}_2$ - 
the tensor product of two copies of the Fock space. There is a standard orthonormal basis of $\mathcal{F}$ given by 
the set $\{\ | m,n \rangle \ \ |\ m,n \in \mathbb{Z},\ m,n \geq 0\}.$ On vectors of this set $c_i$'s act as follows:
\begin{equation*}
c_1(| m,n \rangle ) = \sqrt{m+1}| m+1,n \rangle;\ \ c_2(| m,n \rangle ) = \sqrt{n+1}| m,n+1 \rangle;
\end{equation*}
\begin{equation*} 
c_1^*(| m,n \rangle ) = \sqrt{m}| m-1,n \rangle,\ {\rm if}\ m \geq 1;\ c_1^*(| 0,n \rangle ) = 0;
\end{equation*}
\begin{equation*} 
c_2^*(| m,n \rangle ) = \sqrt{n}| m,n-1 \rangle,\ {\rm if}\ n \geq 1;\ c_2^*(| m,0 \rangle ) = 0.
\end{equation*}
When one uses the representation of equations (\ref{canonical}), one takes 
$\mathcal{F}_1 \cong \mathcal{F}_2 \subset L^2(\mathbb{R})$ and thus 
$\mathcal{F} \subset L^2(\mathbb{R}^2)$. The vacuum state is given by 
${\displaystyle |0,0 \rangle = \frac{1}{2 \pi}{\rm e}^{-(x_1^2 + x_3^2)/2}}.$ 
The excited states of $c_i$'s are given by the Hermite polynomials of $x_1$ and $x_3$, multiplied by 
$\frac{1}{2 \pi}{\rm e}^{-(x_1^2 + x_3^2)/2}.$ In fact, it turns out that 
$\mathcal{F}_1 \cong \mathcal{F}_2 \cong \mathcal{S}(\mathbb{R})$ and $\mathcal{F} \cong \mathcal{S}(\mathbb{R}^2)$.
\par
Instead, we can take as creation operators
\begin{equation*} 
z_1 = x_1 + i x_2 \ \ {\rm  and} \ \ z_2 = x_3 + i x_4,
\end{equation*}
and as annihilation operators
\begin{equation*} 
\frac{\partial}{\partial z_1} \equiv \frac{1}{2}(\frac{\partial}{\partial x_1} - i\frac{\partial}{\partial x_2})
\ \ {\rm  and} \ \ \frac{\partial}{\partial z_2} \equiv 
\frac{1}{2}(\frac{\partial}{\partial x_3} - i\frac{\partial}{\partial x_4}).
\end{equation*}
These operators can be implemented on the Segal-Bargman space $\mathcal{H} L^2(\mathbb{C}^2, \mu_1)$ of all 
holomorphic functions on $\mathbb{C}^2$, $f$, for which $\| f \|_1 < \infty$. The scalar product (and the norm) is 
given by:
\begin{equation*}
\langle f, g \rangle_1 \equiv \frac{1}{\pi^2} \int_{\mathbb{C}^2} \bar{f} g {\rm e}^{-|z_1|^2 - |z_2|^2} dz_1 dz_2.
\end{equation*}
Then in the above notations one has ${\displaystyle |m,n \rangle \ \equiv \ \frac{z_1^m z_2^n}{\sqrt{m  ! n  ! }}}.$

\section{Vector Bundles, Connections and Curvatures} \label{sec3}

Let $F$ be a (principal or complex vector) bundle with structure group $G$ over a manifold $M$ and let 
$\pi : F \to M$ be the bundle projection. The bundle can be given by a set of local trivializations and 
transformations between them which are elements of $G$. 
The set of sections $\Gamma(M,F)$ is the set of all (continuous, differentiable, etc.) 
maps $\gamma : M \to F$, such that $\gamma (x) \in \pi^{-1}(x)$. It follows from the Serre-Swan theorem \cite{S62} that 
every vector bundle on a locally compact manifold is projective (i.e. a summand in a free bundle). It is important to 
note that $\Gamma(M,F)$ is a module over $C^{\infty}(M,\mathbb{C})$.
\par
In the noncommutative case one does not have 'points' because of nonlocality. Therefore in order to define 'bundles', 
one needs to consider (right) modules.
\par
Let \cite{C80} $A$ be a $C^*$-algebra and let $\Lambda$ be a Lie group, acting on $A$ via derivations 
(Lie$\Lambda$ is embedded in the algebra of derivations of $A$). 
Let $A^{\infty}$ be a dense $*$-subalgebra of smooth (with respect to Lie$\Lambda$) elements. 
To each right $A$-module $\Xi$ corresponds a unique (up to isomorphism) right $A^{\infty}$-module $\Xi^{\infty}$, 
for which $\Xi = \Xi^{\infty} \otimes_{A^{\infty}} A$. A right $A^{\infty}$-module $\Xi^{\infty}$ is projective if 
there exists a right $A^{\infty}$-module $\tilde{\Xi}^{\infty}$, such that 
$\Xi^{\infty} \oplus \tilde{\Xi}^{\infty} \cong \mathbb{C}^n \otimes A^{\infty}$ for some $n \in \mathbb{N}$.
\par
A connection on $\Xi^{\infty}$ is a $\mathbb{C}$-linear map 
$\nabla : \Xi^{\infty} \to \Xi^{\infty} \otimes ({\rm Lie}\Lambda)^*,$ satisfying the condition 
$\nabla_X(\xi \cdot x) = \nabla(\xi) \cdot x + \xi \cdot \delta_X(x),$ where 
$X \in {\rm Lie}\Lambda,$ $\delta_X$ is the corresponding element from the algebra of derivations, 
$x \in A^{\infty}$, and $\xi \in \Xi^{\infty}$. The curvature 
$\Theta \in {\rm End}_{A^{\infty}}(\Xi^{\infty}) \otimes \wedge^2({\rm Lie}\Lambda)^*$ of $\nabla$ is given by 
\begin{equation*} 
\Theta (X,Y) = \nabla_X \nabla_Y - \nabla_Y \nabla_X - \nabla_{[X,Y]} \in {\rm End}_{A^{\infty}}(\Xi^{\infty}),\ \ \ 
X,Y \in {\rm Lie}\Lambda.
\end{equation*}
\par
In the $\mathbb{R}^4_{\theta}$ case we have that $\Lambda = \mathbb{R}^4$ is commutative. If we denote 
$F_{jk} = \Theta(\partial_j, \partial_k)$ then we will have $F_{jk} = \nabla_j \nabla_k - \nabla_k \nabla_j$. 
Also in our case $A = \mathcal{A} = \mathcal{K}$ and $A^{\infty}$ is one of the algebras from section \ref{sec2}. 
Below we drop the $\infty$ signs for shortness.
\par
Every connection on $\mathbb{R}^4_{\theta}$ is of the form $\nabla_j = \partial_j + A_j$, where 
$A_j \in {\rm End}_{A}(\Xi).$ Its curvature is 
\begin{equation*} 
F_{mn} = \partial_m A_n - \partial_n A_m + [A_m, A_n].
\end{equation*}
Let's consider a connection of the form $P \partial_j$, where $P$ is the projection onto $\Xi$ in 
$\Xi \oplus \tilde{\Xi} \cong \mathbb{C}^n \otimes A$. In the commutative case \cite{A79} one can find 
local linear maps $u$ for which $P = u u^*$ and $u^* u = 1$. In the noncommutative case this is possible only if:

\begin{lemma} \label{free}
Assume that $U \in M_{n \times k}(\tilde{A})$ is such that $U U^* = P$ and $U^* U = 1_k$. Then the module 
$P M_{n \times n}(\tilde{A})$ is free.
\end{lemma}
\begin{proof}
If we denote by $u_1,\ u_2,\ \dots,\ u_k$ the columns of $U$, then we will have the following isomorphism 
$\iota : {\rm Span}_{\tilde{A}} \{ u_1, \dots, u_k \} \to \mathbb{C}^k \otimes \tilde{A}$, given by: 
\begin{equation*}
\iota : u_1 \cdot a_1 + \dots u_k \cdot a_k \mapsto (a_1,\dots,a_k).
\end{equation*}
The coefficients can be uniquely recovered by multiplying 
$u_1 \cdot a_1 + \dots u_k \cdot a_k = U (a_1, \dots, a_k)^T$ on the left by $U^*$. \\ 
Next we have 
\begin{equation*}
P M_{n \times n}(\tilde{A}) = U U^* M_{n \times n}(\tilde{A}) = 
U U^* U M_{k \times k}(\tilde{A}) = 
\end{equation*}
\begin{equation*}
= U 1_k M_{k \times k}(\tilde{A}) = U M_{k \times k}(\tilde{A}) = 
{\rm Span}_{\tilde{A}} \{ u_1, \dots, u_k \}^k,
\end{equation*}
the range projection of $U^*$ being $1_k$.
\end{proof}

In this case the connection is 
\begin{equation} \label{AU}
A = U^* dU,\ {\rm or} \ A_j = U^* \partial_j U \ {\rm \ and \ satisfies} \ A = -(dU^*) U, 
\end{equation}
because of the relation $U^* U = 1_k$. Its curvature is 
\begin{equation*}
F_{mn} = (\partial_m U^*) \partial_n U + U^* \partial_m \partial_n U - 
(\partial_n U^*) \partial_m U - U^* \partial_n \partial_m U
\end{equation*}
\begin{equation*}
- (\partial_m U^*) U U^* \partial_n U + (\partial_n U^*) U U^* \partial_m U = 
\end{equation*}
\begin{equation*}
(\partial_m U^*) \partial_n U - (\partial_n U^*) \partial_m U + (\partial_n U^*) U U^* \partial_m U - 
(\partial_m U^*) U U^* \partial_n U = 
\end{equation*}
\begin{equation*}
(\partial_m U^*) (1-UU^*) \partial_n U - (\partial_n U^*) (1-UU^*) \partial_m U,
\end{equation*}
or
\begin{equation*}
F_{mn} = (\partial_m U^*) (1-UU^*) \partial_n U - (\partial_n U^*) (1-UU^*) \partial_m U.
\end{equation*}
We note that $U^*(1-UU^*) = (1-UU^*)U = 0$ and therefore if denoted $Q = (1-UU^*)$ we would have  
\begin{equation} \label{Q}
Q^2 = (1-UU^*)(1-UU^*) = 1 - UU^* + (1-UU^*)UU^* = 1 - UU^* = Q.
\end{equation}
\par
We have: 
\begin{equation*}
(\partial_1 U^*) (1-UU^*) \partial_2 U = 
\frac{i}{\theta_{12}} [\hat{x}_2,U^*](1-UU^*)\frac{-i}{\theta_{12}} \ [\hat{x}_1,U] = 
\end{equation*}
\begin{equation*}
\frac{1}{\theta_{12}^2} (\hat{x}_2 U^* - U^* \hat{x}_2) (1-UU^*) (\hat{x}_1 U - U \hat{x}_1) = 
\frac{-1}{\theta_{12}^2} U^* \hat{x}_2 (1-UU^*) \hat{x}_1 U.
\end{equation*}
Analogously 
\begin{equation*}
(\partial_2 U^*) (1-UU^*) \partial_1 U = \frac{-1}{\theta_{12}^2} U^* \hat{x}_1 (1-UU^*) \hat{x}_2 U
\end{equation*}
and therefore 
\begin{equation} \label{F12}
F_{12} = \frac{1}{\theta_{12}^2} U^*( \hat{x}_1 (1-UU^*) \hat{x}_2 - \hat{x}_2 (1-UU^*) \hat{x}_1 )U.
\end{equation}
Further
\begin{equation} \label{F13}
F_{13} = \frac{1}{\theta_{12} \theta_{34}} U^*( \hat{x}_2 (1-UU^*) \hat{x}_4 - \hat{x}_4 (1-UU^*) \hat{x}_2 )U,
\end{equation}
\begin{equation} \label{F14}
F_{14} = \frac{1}{\theta_{12} \theta_{34}} U^*( \hat{x}_3 (1-UU^*) \hat{x}_2 - \hat{x}_2 (1-UU^*) \hat{x}_3 )U,
\end{equation}
\begin{equation} \label{F23}
F_{23} = \frac{1}{\theta_{12} \theta_{34}} U^*( \hat{x}_4 (1-UU^*) \hat{x}_1 - \hat{x}_1 (1-UU^*) \hat{x}_4 )U,
\end{equation}
\begin{equation} \label{F24}
F_{24} = \frac{1}{\theta_{12} \theta_{34}} U^*( \hat{x}_1 (1-UU^*) \hat{x}_3 - \hat{x}_3 (1-UU^*) \hat{x}_1 )U,
\end{equation}
\begin{equation} \label{F34}
F_{34} = \frac{1}{\theta_{34}^2} U^*( \hat{x}_3 (1-UU^*) \hat{x}_4 - \hat{x}_4 (1-UU^*) \hat{x}_3 )U.
\end{equation}

All projective modules over $\tilde{\mathcal{A}}$ are of the form \cite{S02, KS01} 
$\mathcal{E}_{kn} \equiv {\mathcal{F}^*}^k \oplus \tilde{\mathcal{A}}^n,$ where $n,k \in \mathbb{N}.$ 
This follows essentially from equations (\ref{K}). \\ 
Every connection on $\mathcal{E}_{kn}$ is of the form
\begin{equation*} 
\nabla_j = \left(\begin{array}{ll} i\theta_{jl}^{-1}\hat{x}_l & 0 \\ 0 & \partial_j \end{array}\right) +
\left(\begin{array}{ll} \ C_j & | B_j \rangle \\ \langle E_j | & \ \hat{D_j} \end{array}\right),
\end{equation*}
where $C_j \in M_{kk}(\mathbb{C}),\ | B_j \rangle, | E_j \rangle \in M_{kn}(\mathcal{F}),\ 
\hat{D_j} \in M_{nn}(\tilde{\mathcal{A}}),$ and $ \langle E_j | = | E_j \rangle^*.$ \\ 
We also note that the curvature of 
\begin{equation*}
\bar{\nabla}_j \equiv \left(\begin{array}{ll} i\theta_{jl}^{-1}\hat{x}_l & 0 \\ 0 & \partial_j \end{array}\right)
\end{equation*}
is $F_{jl} = -i\theta^{-1}_{jl}.$
\par
The operator trace of endomorphisms of such modules is given by \cite{KS01}
\begin{equation} \label{trace1}
\text{Tr}\left(\begin{array}{ll} \ C & | B \rangle \\ \langle E | & \ \hat{D} \end{array}\right) = 
(2\pi)^2(\theta_1 \theta_2) \underset{l=1}{\overset{k}{\sum}} C_{ll} + \underset{l=1}{\overset{n}{\sum}} 
\int D_{ll}(x),
\end{equation}
where $D_{ab}(x)$ is the corresponding function, obtained by the reverse of the Weyl transform.
\par
Every projective $\mathcal{A}$-module is of the form $P \mathcal{A}^N$, for some projection 
$P \in M_N(\mathcal{A})$. For unital algebras we have ${\rm End}(P \tilde{\mathcal{A}}^N) = 
P M_N(\tilde{\mathcal{A}}) P$ (for nonunital algebras this is not true). 
The operator trace on $M_N(\mathcal{A})$ is: \cite{TZS02}
\begin{equation*}
{\rm Tr}(\star) \equiv \int {\rm Tr}_N(\star) d^4x = 
\pi^2 | \theta_{12} \theta_{34} | {\rm Tr}_{\mathcal{F}}( {\rm Tr}_N(\star) ) =
\end{equation*}
\begin{equation} \label{trace2} 
 = \pi^2 | \theta_{12} \theta_{34} | \sum^{\infty}_{n_1, n_2 = 0} 
\langle n_1, n_2 | {\rm Tr}_N(\star) | n_1, n_2 \rangle.
\end{equation}

\section{Fields, Gauges, Instantons, ASD equations}

The Yang-Mills action is given by
\begin{equation} \label{YM}
S = \frac{1}{g^2} {\rm Tr} \frac{1}{4} (F_{ij}F_{ij}) d^4x,
\end{equation}
where {\rm Tr} is given by equation (\ref{trace1}) or equation (\ref{trace2}), respectively.
Minimizing the action yields in a standard way the ASD equations:
\begin{equation} \label{ASD}
F_{12} = - F_{34},\ \ F_{13} = F_{24},\ \ F_{14} = - F_{23}.
\end{equation}
\par
Let $\mathcal{A}$ be an algebra from Section \ref{sec2}. 
Usually gauge theories over submodules of $\tilde{\mathcal{A}}^N$ are considered as $U(N)$ gauge theories. 
For our purposes we define '$U(N)$ gauge group' as
\begin{equation*}
{\displaystyle U^N_0(\tilde{\mathcal{A}}) \equiv \{ {\rm exp}(i H_1) \cdots {\rm exp}(i H_n) \ | 
\ n \in \mathbb{N},\ H_j \in M_N(\tilde{\mathcal{A}}),\ H_j^* = H_j, \ \forall j \}.}
\end{equation*}
Since it is usually supposed that a connection takes values in the Lie algebra of the gauge group, here we will suppose 
that a connection takes values in the set of all antihermitean operators of $M_N(\tilde{\mathcal{A}}).$ Note that this 
choice is consistent with the special case of equation (\ref{AU}).
\par
It is clear that $\partial_j$ is a derivation $\forall j$ (actually it is an inner derivation). 
Taking a $U \in U^N_0(\tilde{\mathcal{A}})$ and applying ${\displaystyle \partial_j}$ to ${\displaystyle UU^* = 1}$ gives 
\begin{equation*}
{\displaystyle \partial_j(1) = \partial_j(UU^*) = \partial_j(U)U^* + U \partial_j(U^*)}.
\end{equation*}
Thus ${\displaystyle U \partial_j(U^*) = - \partial_j(U)U^* }.$ Multiplying by $U^*$ to the left gives 
\begin{equation*} 
{\displaystyle \partial_j(U^*) = - U^* \partial_j(U)U^*}.
\end{equation*}
Under the action of $U$ a connection transforms as
\begin{equation*} 
\nabla_j = \partial_j + A_j \mapsto U^* (\partial_j + A_j) U = \partial_j + U^* A_j U + U^* \partial_j U.
\end{equation*}
A standard computation yields the transformation
\begin{equation*} 
F_{mn} \mapsto U^* F_{mn} U. 
\end{equation*}
In some sense the unitary group $U^N_0(\mathcal{A})$ is the largest possible:
\begin{lemma} \label{U^N_0}
$U^N_0(\tilde{\mathcal{A}})$ is norm dense in the group $U^N(\tilde{\mathcal{K}})$ of all unitary operators from 
$M_N(\tilde{\mathcal{K}}).$  
\end{lemma}
\begin{proof}
$U^N_0(\tilde{\mathcal{A}})$ is norm-dense in $U^N_0(\tilde{\mathcal{K}})$. From Proposition 2.1.6 of \cite{RLL00} 
follows that $U^N_0(\tilde{\mathcal{K}})$ is the connected component of the identity operator. 
But $M_N(\tilde{\mathcal{K}})$ is stable, because 
$M_N(\tilde{\mathcal{K}}) \otimes \tilde{\mathcal{K}} \cong 
M_N(\tilde{\mathcal{K}} \otimes \tilde{\mathcal{K}}) \cong M_N(\tilde{\mathcal{K}})$ (the $C^*$-norm on 
$\tilde{\mathcal{K}} \otimes \tilde{\mathcal{K}}$ is unique). Exercise 8.17 of \cite{RLL00} gives 
\begin{equation*}
U^N(\tilde{\mathcal{K}})/U^N_0(\tilde{\mathcal{K}}) \cong 
U(M_N(\tilde{\mathcal{K}}))/U_0(M_N(\tilde{\mathcal{K}})) \cong K_1(M_N(\tilde{\mathcal{K}})) \cong 
K_1(\tilde{\mathcal{K}}) \cong \{ 0 \}.
\end{equation*}
The last equality is from equation (\ref{K}).
\end{proof}
\par
A submodule of $\mathcal{A}^N$ is sent to an isomorphic submodule. If we consider a theory  on such a submodule as a 
Yang-Mills theory on $\mathcal{A}^N$, then the action (\ref{YM}) will be preserved when acting with elements of 
$U^N_0(\tilde{\mathcal{A}}).$ 
\par
It is clear that an ASD-connsection (i.e. an instanton) is mapped to an ASD-connsection. This follows from the fact 
that the transformed connection will still satisfy equations (\ref{ASD}). Moreover it is mapped to a connection of the 
same topological number. The topological number is defined as: \cite{TZS02}
\begin{equation*}
Q \equiv \frac{1}{16 \pi^2} {\rm Tr} (F_{mn}F_{mn}) = \frac{g^2}{4 \pi^2} S.
\end{equation*}
\par
In analogy we can consider antihermitean connections on the modules $\mathcal{E}_{kn}$: 
\begin{equation} \label{ahA}
\nabla_j = \left(\begin{array}{ll} \tilde{\partial_j} & 0 \\ 0 & \partial_j \end{array}\right) +
i \left(\begin{array}{ll} \ C_j & | B_j \rangle \\ \langle B_j | & \ \hat{D_j} \end{array}\right),
\end{equation}
where $C_j^* = C_j$, $\hat{D_j}^* = \hat{D_j}$, $\langle B_j | = | B_j \rangle^*$ and 
$\tilde{\partial_j} \equiv i\theta_{jl}^{-1}\hat{x}_l.$
In \cite{S01,KS01} a different gauge condition has been considered. 
Applying ASD equations (\ref{ASD}) to (\ref{ahA}) twelve equations are obtained for the twelve unknown operators of 
the $A_j$'s. 
\begin{lemma} \label{curvatureahA}
The curvature of (\ref{ahA}) is
\begin{equation*} 
(F_{mn})_{11} = i \tilde{\partial_m} C_n - i \tilde{\partial_n} C_m + C_n C_m - C_m C_n + | B_n \rangle \langle B_m | - 
| B_m \rangle \langle B_n |,
\end{equation*}
\begin{equation*} 
(F_{mn})_{12} = i \tilde{\partial_m} | B_n \rangle - i \tilde{\partial_n} | B_m \rangle + C_n | B_m \rangle - 
C_m | B_n \rangle + | B_n \rangle \hat{D_m} - | B_m \rangle \hat{D_n},
\end{equation*}
\begin{equation*} 
(F_{mn})_{21} = i \partial_m \langle B_n | - i \partial_n \langle B_m | + \langle B_n | C_m - \langle B_m | C_n + 
\hat{D_n} \langle B_m | - \hat{D_m} \langle B_n |,
\end{equation*}
\begin{equation*} 
(F_{mn})_{22} = i \partial_m \hat{D_n} - i \partial_n \hat{D_m} + \langle B_n | B_m \rangle - 
\langle B_m | B_n \rangle + \hat{D_n} \hat{D_m} - \hat{D_m} \hat{D_n}.
\end{equation*}
\end{lemma}
\hfill $\Box$

When we work with projective modules which are not free, we must use a more general connection of the form $A=\Pi d\Pi,$
where $\Pi$ is a projection on a free module.

\section{ASD equations on free modules}

In this section, for the sake of simplicity, we will assume that $\theta_{12} = \theta_{34} = 1.$
\par
Plugging in equations (\ref{F12}) - (\ref{F34}) into equation (\ref{ASD}) we get 

\begin{equation} \label{F12=-F34}
U^*( \hat{x}_1 (1-UU^*) \hat{x}_2 - \hat{x}_2 (1-UU^*) \hat{x}_1 )U = 
U^*( \hat{x}_4 (1-UU^*) \hat{x}_3 - \hat{x}_3 (1-UU^*) \hat{x}_4 )U,
\end{equation}
\begin{equation} \label{F13=F24}
U^*( \hat{x}_2 (1-UU^*) \hat{x}_4 - \hat{x}_4 (1-UU^*) \hat{x}_2 )U = 
U^*( \hat{x}_1 (1-UU^*) \hat{x}_3 - \hat{x}_3 (1-UU^*) \hat{x}_1 )U,
\end{equation}
\begin{equation} \label{F14=-F23}
U^*( \hat{x}_3 (1-UU^*) \hat{x}_2 - \hat{x}_2 (1-UU^*) \hat{x}_3 )U = 
U^*( \hat{x}_1 (1-UU^*) \hat{x}_4 - \hat{x}_4 (1-UU^*) \hat{x}_1 )U.
\end{equation}

Setting $1-UU^* = Q$ and making use of equation (\ref{Q}), we see that equations 
(\ref{F12=-F34}) - (\ref{F14=-F23}) are equivalent to 
\begin{equation} \label{MF12=-F34}
U^* \left(\!\!\begin{array}{cccc} \sigma_1 x_1 & \sigma_2 x_2 & \sigma_3 x_3 & x_4 \end{array}\!\!\right) Q 
\left(\!\!\begin{array}{cccc}0 & 1 & 0 & 0 \\ 1 & 0 & 0 & 0 \\
0 & 0 & 0 & i \\ 0 & 0 & -i & 0 \end{array}\!\!\right)
Q \left(\!\!\begin{array}{c} \sigma_1 x_1 \\ \sigma_2 x_2 \\ \sigma_3 x_3 \\ x_4 \end{array}\!\!\right) U = 0,
\end{equation}
\begin{equation} \label{MF13=F24}
U^* \left(\!\!\begin{array}{cccc} \sigma_1 x_1 & \sigma_2 x_2 & \sigma_3 x_3 & x_4 \end{array}\!\!\right) Q 
\left(\!\!\begin{array}{cccc} 0 & 0 & 1 & 0 \\ 0 & 0 & 0 & i \\
1 & 0 & 0 & 0 \\ 0 & -i & 0 & 0 \end{array}\!\!\right)
Q \left(\!\!\begin{array}{c} \sigma_1 x_1 \\ \sigma_2 x_2 \\ \sigma_3 x_3 \\ x_4 \end{array}\!\!\right) U = 0,
\end{equation}
\begin{equation} \label{MF14=-F23}
U^* \left(\!\!\begin{array}{cccc} \sigma_1 x_1 & \sigma_2 x_2 & \sigma_3 x_3 & x_4 \end{array}\!\!\right) Q 
\left(\!\!\begin{array}{cccc}0 & 0 & 0 & -i \\ 0 & 0 & -1 & 0 \\
0 & -1 & 0 & 0 \\ i & 0 & 0 & 0 \end{array}\!\!\right)
Q \left(\!\!\begin{array}{c} \sigma_1 x_1 \\ \sigma_2 x_2 \\ \sigma_3 x_3 \\ x_4 \end{array}\!\!\right) U = 0,
\end{equation}
where for shortness we have written $U = U \otimes {\rm Id}_2,\ Q = Q \otimes {\rm Id}_2$ and $U^* = U^* \otimes {\rm Id}_2.$
Here $\sigma_1, \sigma_2, \sigma_3$ are the Pauli matrices, they satisfy $\sigma_1 \sigma_2 = i \sigma_3$, 
$\sigma_2 \sigma_3 = i \sigma_1$, $\sigma_3 \sigma_1 = i \sigma_2$, and anticommute: 
$\sigma_a \sigma_b = 2 \delta_{ab} {\rm Id}.$ Summing these three equations (\ref{MF12=-F34}), (\ref{MF13=F24}) 
and (\ref{MF14=-F23}) we get that (\ref{ASD}) is equivalent to the equation
\begin{equation*} 
U^* \left(\!\!\begin{array}{cccc} \sigma_1 x_1 & \sigma_2 x_2 & \sigma_3 x_3 & x_4 \end{array}\!\!\right) Q 
\left(\!\!\begin{array}{cccc}0 & 1 & 1 & -i \\ 1 & 0 & -1 & i \\
1 & -1 & 0 & i \\ i & -i & -i & 0 \end{array}\!\!\right)
Q \left(\!\!\begin{array}{c} \sigma_1 x_1 \\ \sigma_2 x_2 \\ \sigma_3 x_3 \\ x_4 \end{array}\!\!\right) U = 0.
\end{equation*}
If we denoted 
\begin{equation*} 
\Xi = \left(\!\!\begin{array}{cccc}0 & 1 & 1 & -i \\ 1 & 0 & -1 & i \\
1 & -1 & 0 & i \\ i & -i & -i & 0 \end{array}\!\!\right)
\text{ \ \ and \ \ }
\Omega = \left(
\begin{array}{cccc}
 \frac{1}{2} & -\frac{1}{2} & -\frac{1}{2} & \frac{i}{2} \\
 \frac{1}{\sqrt{2}} & 0 & 0 & -\frac{i}{\sqrt{2}} \\
 \frac{1}{\sqrt{6}} & \sqrt{\frac{2}{3}} & 0 & \frac{i}{\sqrt{6}} \\
 \frac{1}{2 \sqrt{3}} & -\frac{1}{2 \sqrt{3}} & \frac{\sqrt{3}}{2} & \frac{i}{2 \sqrt{3}} \\
\end{array}
\right)
\end{equation*}
we would have
\begin{equation*} 
\Xi \ = \ \Omega^* \ \left(\!\!\begin{array}{cccc}-3 & 0 & 0 & 0 \\ 0 & 1 & 0 & 0 \\
0 & 0 & 1 & 0 \\ 0 & 0 & 0 & 1 \end{array}\!\!\right)\ \Omega.
\end{equation*}
Thus we can put 
\begin{equation*} 
p_0 = \frac{\sigma_1 x_1}{2}-\frac{\sigma_2 x_2}{2}-\frac{\sigma_3 x_3}{2}+\frac{i x_4}{2}, \ 
p_1 = \frac{\sigma_1 x_1}{\sqrt{2}}-\frac{i x_4}{\sqrt{2}},
\end{equation*}
\begin{equation*} 
p_2 = \frac{\sigma_1 x_1}{\sqrt{6}}+\sqrt{\frac{2}{3}} \sigma_2 x_2 +\frac{i x_4}{\sqrt{6}}, \ 
p_3 = \frac{\sigma_1 x_1}{2 \sqrt{3}}-\frac{\sigma_2 x_2}{2 \sqrt{3}}+\frac{\sqrt{3} \sigma_3 x_3}{2}+\frac{i x_4}{2 \sqrt{3}}
\end{equation*}
and obtain 
\begin{proposition} 
The ASD equations (\ref{ASD}) are equivalent to the equation
\begin{equation} \label{p0p1p2p3}
U^* [ -3p_0^* Q p_0  +  p_1^* Q p_1  +  p_2^* Q p_2  +  p_3^* Q p_3 ] U = 0.
\end{equation}
\end{proposition}

Note that (\ref{p0p1p2p3}) is a Schr\"{o}dinger-like equation, due to the fact that the presence of $x_i$'s in the 
positive operators $U^* p_0^* Q p_0 U$, $U^* p_1^* Q p_1 U$, $U^* p_2^* Q p_2 U$, and $U^* p_3^* Q p_3 U$ corresponds 
to taking derivatives.

\section{ADHM construction and topological charge}

We now recall the ADHM-construction (ADHM stands for Atiyah, Drinfeld, Hitchin and Manin). In the conventional case 
it has been introduced in \cite{ADHM78}. The analogue in the noncommutative case has been done in \cite{NS98}.
\par
Consider \cite{TZS02} complex vector spaces $V$ and $W$, 
${\rm dim}_{\mathbb{C}} V = k,\ {\rm dim}_{\mathbb{C}} W = n$. 
Let $B_1, B_2 \in M_{kk}(\mathbb{C}),\ I \in M_{kn}(\mathbb{C}),\ J \in M_{nk}(\mathbb{C})$ be matrices, satisfying 
\begin{equation} \label{ADHM1}
[B_1,B_1^*] + [B_2, B_2^*] + II^* - J^*J = 2(\theta_{12} + \theta_{34}) {\rm Id}_n
\end{equation}
and
\begin{equation} \label{ADHM2}
[B_1,B_2] + IJ = 0.
\end{equation}
Next consider the operator
\begin{equation*} 
\Delta^* = \left(\begin{array}{lll} I & B_2 + \hat{z}_2 & B_1 + \hat{z}_1 \\
J^* & -B_1^* - \hat{z}_1^* & B_2^* + \hat{z}_2^* \end{array}\right),
\end{equation*}
where $\hat{z}_1 = \hat{x_2} + i \hat{x_1} \  {\rm  and} \  \hat{z}_2 = \hat{x_4} + i \hat{x_3}.$ Conditions 
(\ref{ADHM1}) and (\ref{ADHM2}) are equivalent to 
\begin{equation*} 
\Delta^* \Delta = \left(\begin{array}{ll} \Gamma & 0  \\ 0 & \Gamma \end{array}\right),
\end{equation*}
where $\Gamma \in \mathbb{A}$ is some operator, where 
$\mathbb{A} \equiv {\rm Alg}(1, \hat{x_1}, \hat{x_2}, \hat{x_3}, \hat{x_4}).$ 
It turns out that \cite{N00, S01} due to equations (\ref{ADHM1}) and (\ref{ADHM2}) $\Delta^* \Delta$ has no kernel. 
Therefore it is formally invertible. 
If we denoted by $\Upsilon$ the right $\mathbb{A}$-module generated by $(\Delta^* \Delta)^{-1/2} \Delta^*$, i.e. 
$\Upsilon = (\Delta^{-1} \Delta)^{-1/2} \mathbb{A}$, then it would turn out that 
$\Pi' = \Delta (\Delta^* \Delta)^{-1} \Delta^*$ would be the projection onto $\Upsilon$. 
Therefore $\Pi = 1 - \Pi'$ would be the projection onto the orthogonal complement $\Psi$ of $\Upsilon$ in the module 
$(V \oplus V \oplus W) \otimes \mathbb{A}$ with respect to the inner product 
$\langle (a_1, \dots, a_{2k + n}),\ (b_1, \dots, b_{2k + n}) \rangle = a_1^* b_1 + \dots + a_{2k + n}^* b_{2k + n}.$
\par
From now on we assume that $\theta_{12} = \theta_{34} = 1$.
\par
The solutions of $\Delta^* U = 0$ can be combined into a rank $n$ free module part $V$, consisting of $n$ columns, 
and a part $K$, consisting of $k$ vectors from $\mathcal{K}$. The zero modes of 
\begin{equation} \label{zero}
\left(\begin{array}{ll}  B_2 + \hat{z}_2 & B_1 + \hat{z}_1 \\
 -B_1^* - \hat{z}_1^* & B_2^* + \hat{z}_2^* \end{array}\right)
\end{equation}
comprise of a $k$ dimensional space, because (\ref{zero}) is in fact an elliptic pseudodifferential operator of 
order $k$, and therefore \cite{S01} has index $k$. \\
If $|v_1 \rangle, \dots, |v_k \rangle$ is an orthonormal family of zero modes of (\ref{zero}) we can set 
\begin{equation*}
K_{1,(n+1)} = |v_1 \rangle \langle 0,0| + \dots |v_k \rangle \langle 0,k-1| 
\end{equation*}
and all the other entries of $K \in M_{(n+2k),(n+2k)}$ equal to zero. Then we can write $U = V + K$ and have:
\begin{equation*}
U^* U = 1_n \oplus \mathbb{I}_k,\ \ UU^* = \Pi,\ \ V^*V = 1_n,\ \ K^*K = \mathbb{I}_k,
\end{equation*}
where $\mathbb{I}_k$ is the projection of $M_{(n+2k),(n+2k)}$ with all entries equal to zero, except the 
$(n+1) \times (n+1)$-st entry, which is $|0,0 \rangle \langle 0,0| + \dots |0,k-1 \rangle \langle 0,k-1|.$
\par
We now want to find the connection $\nabla = d + A$ of the module $\Psi$. We have \cite{A79} 
\begin{equation*}
\nabla(U \psi) = \Pi d(U \psi) = UU^* d(U \psi) = U[U^* (dU) \psi + U^* U d \psi] = 
\end{equation*}
\begin{equation*}
U[U^* (dU) \psi + (1_n \oplus \mathbb{I}_k)(d \psi)].
\end{equation*}
Therefore $A = U^* dU.$
\par
The curvature is 
\begin{equation*}
F_{pq} = \partial_p U^* \partial_q U - \partial_q U^* \partial_p U + 
(U^* \partial_p U)(U^* \partial_q U) - (U^* \partial_q U)(U^* \partial_p U) = 
\end{equation*}
\begin{equation*}
\partial_p U^* (1-UU^*) \partial_q U - \partial_q U^* (1-UU^*) \partial_p U + 
(\partial_p \mathbb{I}_k) U^* \partial_q U - (\partial_q \mathbb{I}_k) U^* \partial_p U.
\end{equation*}
\par
The computation of the topological index of the ADHM connection in \cite{TZS02} uses the assumption that 
$U^*U = 1_n$, which is incorrect. Thus in the curvature the term 
$(\partial_p \mathbb{I}_k) U^* \partial_q U - (\partial_q \mathbb{I}_k) U^* \partial_p U$ has been omitted. The 
computation uses the Corrigan's identity 
\begin{equation*}
{\rm Tr}_{2k+n}(F_{pq} F_{pq}) = \frac{1}{2} \partial_q \partial_q {\rm Tr}[\sigma_p b^* 
(2 - \Delta \Gamma^{-1} \Delta^*) b \bar{\sigma_p} \Gamma^{-1}],
\end{equation*}
where
\begin{equation*}
b = \left(\begin{array}{ll}  0 & 0  \\ 1 & 0 \\ 0 & 1 \end{array}\right).
\end{equation*}
Its derivation \cite{DKM96} does not depend on the assumption $U^*U = 1_n$, and therefore a correction is due only to 
the analysis of the missing term in the curvature. In \cite{S02} the analysis of the topological index does not use 
the Corrigan's identiry and is done more carefully.
\par
Further we obtain
\begin{equation*}
F_{pq} F_{pq} = [\partial_p U^* (1 - UU^*) \partial_q U - \partial_q U^* (1 - UU^*) \partial_p U]^2 + 
\end{equation*}
\begin{equation*}
[\partial_pU^* (1- UU^*) \partial_q U - \partial_q U^* (1- UU^*) \partial_p U] 
[(\partial_p \mathbb{I}_k) U^* \partial_q U - (\partial_q \mathbb{I}_k) U^* \partial_p U] + 
\end{equation*}
\begin{equation*}
[(\partial_p \mathbb{I}_k) U^* \partial_q U - (\partial_q \mathbb{I}_k) U^* \partial_p U]
[\partial_pU^* (1- UU^*) \partial_q U - \partial_q U^* (1- UU^*) \partial_p U] +
\end{equation*}
\begin{equation} \label{FF}
[(\partial_p \mathbb{I}_k) U^* \partial_q U - (\partial_q \mathbb{I}_k) U^* \partial_p U]^2.
\end{equation} 
The first term is the one that contributes to the index. We will show that the other terms give zero in 
${\rm Tr}_{\mathcal{H}} ({\rm Tr}_{2k+n}( \cdot )).$ 
\par
For those terms we need the following formulae:
\begin{equation*}
\partial_1(\mathbb{I}_k) = \partial_2(\mathbb{I}_k) = 0,
\end{equation*}
\begin{equation*}
\partial_3(\mathbb{I}_k) = \frac{k}{2i} (|0,k \rangle \langle 0, k-1| + |0, k-1 \rangle \langle 0, k|,
\end{equation*}
\begin{equation*}
\partial_4(\mathbb{I}_k) = \frac{k}{2} (|0,k \rangle \langle 0, k-1| - |0, k-1 \rangle \langle 0, k|,
\end{equation*}
\begin{equation*}
\partial_3 K = \sum_{j=0}^{k-1} \ \{ \ \frac{z_2-z_2^*}{2i} |v_j \rangle \langle j| - 
\frac{j}{2i} |v_j \rangle \langle j-1| + \frac{j+1}{2i} |v_j \rangle \langle j+1 | \ \},
\end{equation*}
\begin{equation*}
\partial_4 K = \sum_{j=0}^{k-1} \ \{ \ \frac{z_2+z_2^*}{2} |v_j \rangle \langle j| - 
\frac{j}{2} |v_j \rangle \langle j-1| - \frac{j+1}{2} |v_j \rangle \langle j+1 | \ \},
\end{equation*}
\par
The eventual nonzero contributions can arize only for $p=3,\ q=4$. The case $p=4,\ q=3$ is symmetric.
The second and the third terms of (\ref{FF}) are equal by the trace property. Clearly the only nonzero element 
of $\partial_3 U (\partial_3 \mathbb{I}_k) U^*$ is 
${\displaystyle \frac{k}{2i} \frac{k}{2i} |0,k \rangle \langle v_k |}$ and in 
$\partial_4 U (\partial_4 \mathbb{I}_k) U^*$ it is 
${\displaystyle \frac{-k}{2} \frac{k}{2i} |0,k \rangle \langle v_k |}$. Therefore 
$(1 - UU^*) \partial_3 U (\partial_3 \mathbb{I}_k) U^* = 0$ and the second and third terms of 
(\ref{FF}) are zero. The last term of (\ref{FF}) is (note that we can write $K$ instead of 
$U$ in this)
\begin{equation*}
[(\partial_3 \mathbb{I}_k) U^* \partial_4 U - (\partial_4 \mathbb{I}_k) U^* \partial_3 U]^2 = 
\end{equation*} 
\begin{equation*}
(\partial_3 \mathbb{I}_k) K^* \partial_4 K (\partial_3 \mathbb{I}_k) K^* \partial_4 K - 
(\partial_3 \mathbb{I}_k) K^* \partial_4 K (\partial_4 \mathbb{I}_k) K^* \partial_3 K - 
\end{equation*} 
\begin{equation*}
(\partial_4 \mathbb{I}_k) K^* \partial_3 K (\partial_3 \mathbb{I}_k) K^* \partial_4 K + 
(\partial_4 \mathbb{I}_k) K^* \partial_3 K (\partial_4 \mathbb{I}_k) K^* \partial_3 K.
\end{equation*} 
The second and the third term have the same contribution by the trace property. 
From the above observations it is easy to see that the nonzero contibution to the trace in the first term is 
\begin{equation*}
\frac{k}{2i} |0,k \rangle \langle 0, k-1| |0, k-1 \rangle \langle v_k | \frac{-k}{2} | v_k \rangle \langle 0, k | 
\frac{k}{2i} |0, k \rangle \langle 0, k-1 | | 0,k-1 \rangle . 
\end{equation*} 
\begin{equation*}
\langle v_k | \frac{-k}{2} | v_k \rangle \langle 0, k | = \frac{-k^4}{16} |0,k \rangle \langle 0,k |.
\end{equation*} 
In completely analogous way we obtain that the only nonzero contribution to the trace in the second term is  
\begin{equation*}
\frac{k^4}{16} |0,k \rangle \langle 0,k |,
\end{equation*}
and in the fourth term is
\begin{equation*}
\frac{-k^4}{16} |0,k \rangle \langle 0,k |.
\end{equation*}
Therefore 
\begin{equation*}
{\rm Tr}_{\mathcal{H}} ({\rm Tr}_{2k+n}( [(\partial_p \mathbb{I}_k) U^* \partial_q U - 
(\partial_q \mathbb{I}_k) U^* \partial_p U]^2)) = -\frac{k^4}{16} + \frac{k^4}{16} + \frac{k^4}{16} - 
\frac{k^4}{16} = 0.
\end{equation*} 
Thus we conclude that the topological index of an ADHM connection is equal to $k$, as proposed.

\end{document}